\documentclass{svproc}

\usepackage{url}

\usepackage{hyperref}

\usepackage{amssymb,amsmath, amsrefs}
 
\newtheorem{thm}{Theorem}

\newtheorem{defi}[thm]{Definition}
\newtheorem{rem}[thm]{Remark}
\newtheorem{nota}[thm]{Notation}
\newtheorem{exa}[thm]{Example}

\newtheorem{ack}[thm]{Acknowledgement}

\newtheorem{conj}[thm]{Conjecture}

\newcommand\be{\begin{equation}}
\newcommand\ee{\end{equation}} 

\usepackage{amsmath,amsfonts}

\def\bdefi{\begin{defi}\rm}
\def\edefi{\end{defi}}
\def\bnota{\begin{nota}\rm}
\def\enota{\end{nota}}
\def\rr{\mathfrak{r}}
\def\bb{\mathfrak{b}}

\def\ZF{\textup{\textsf{ZF}}}

\def\({\textup{(}}
\def\){\textup{)}}

\def\RCAo{\textup{\textsf{RCA}}_{0}^{\omega}}

\def\bye{\end{document}}
\def\N{{\mathbb  N}}
\def\Q{{\mathbb  Q}}
\def\R{{\mathbb  R}}

\def\SS{\textup{\textsf{S}}}

\def\J{\mathcal{J}}

\def\di{\rightarrow}

\def\asa{\leftrightarrow}

\def\AC{\textup{\textsf{AC}}}

\def\NIN{\textup{\textsf{NIN}}}

\def\NFP{\textup{\textsf{NFP}}}

\def\HBU{\textup{\textsf{HBU}}}

\def\lex{\textup{\textsf{lex}}}

\def\eps{\varepsilon}

\usepackage{comment}

\begin{document}
\mainmatter              
\title{Between Turing and Kleene}
\titlerunning{Between Turing and Kleene}  
%
\author{Sam Sanders}
\authorrunning{S.\ Sanders} 
%
\tocauthor{Sam Sanders}
\institute{
Department of Philosophy II, RUB Bochum, \\
Universit\"atsstrasse 150, 44780 Bochum, Germany \\ 
\email{sasander@me.com}}

\setcounter{tocdepth}{3}
\setcounter{page}{0}
\tableofcontents
\thispagestyle{empty}
\newpage

\maketitle              

\begin{abstract}
Turing's famous `machine' model constitutes the first intuitively convincing framework for \emph{computing with real numbers}.  
Kleene's computation schemes S1-S9 extend Turing's approach to \emph{computing with objects of any finite type}.  
Both frameworks have their pros and cons and it is a natural question if there is an approach that marries the best of both the Turing and Kleene worlds.  
In answer to this question, we propose a considerable extension of the scope of Turing's approach.  Central is a fragment of the Axiom of Choice involving \emph{continuous} choice functions, going back to Kreisel-Troelstra and intuitionistic analysis.  Put another way, we formulate a relation `is computationally stronger than' involving \textbf{third-order} objects that overcomes (many of) the pitfalls of the Turing and Kleene frameworks. 
\end{abstract}

\section{Between Turing and Kleene computability}
\subsection{Short summary}\label{sintro}
In a nutshell, we propose a sizable extension of the scope of Turing's `machine' model of computation (\cite{tur37}), motivated by a fragment of the Axiom of Choice involving \emph{continuous} choice functions, going back to Kreisel-Troelstra and intuitionistic analysis (\cite{KT}).  In particular, we formulate a relation `is computationally stronger than' involving \textbf{third-order} objects \emph{but} still based on Turing computability by and large. 

\medskip

The interested reader will find the aforementioned extension discussed in more detail in Section \ref{sketchy}, along with a critical discussion of the scope of our extension. 
%
%
The critical reader will learn about the pressing need for the aforementioned extension in Section \ref{crit}.
In particular, the latter section seeks to alleviate worries that existing frameworks are somehow sufficient for our (foundational) needs.  
The (problems involving the) representation of third-order objects via second-order ones is a particularly important `case in point'.  

\medskip

Next, some elegant results in our proposed extension are listed in Section~\ref{main} pertaining to the following topics:
\begin{itemize}
\item convergence theorems for nets in the unit interval (Section \ref{conetx}),
\item covering theorems for the unit interval $\R$ (Section \ref{cov}),
\item the uncountability of the real numbers $\R$ (Section \ref{flonk}),
\item discontinuous functions on the real numbers $\R$ (Section \ref{disc}).
\end{itemize}
We note that all our results are part of classical mathematics, while we have found 
constructive mathematics highly inspiring on our journey towards this paper.  We will assume familiarity with Turing-style computability theory (\cite{zweer}) 
and higher-order primitive recursion like in G\"odel's system $T$ (\cite[p.\ 74]{longmann}); knowledge of Kleene's higher-order computability theory, in particular the computation schemes S1-S9 (see \cite{kleeneS1S9, longmann}),
is useful but not essential.

\medskip

Finally, we will discuss a number of theorems of real analysis and the following remark discusses how the representations of real numbers can be done in a straightforward and non-intrusive way. 
\begin{rem}[Representation of real numbers]\rm\label{fooker}
Kohlenbach's `hat function' from \cite[p.\ 289]{kohlenbach2} guarantees that every element of $\N^{\N}$ defines a real number via the well-known representation of reals as fast-converging Cauchy sequences.  
Despite the definition of the latter being $\Pi_{1}^{0}$, a quantifier `$(\forall x\in \R)$' amounts to a quantifier over $\N^{\N}$.

\medskip

Moreover, Kohlenbach's `tilde' function from \cite[Def.\ 4.24]{kohlenbach3} guarantees that `$(\forall x\in [0,1])$' also just amounts to a quantifier over $\N^{\N}$, despite $0\leq_{\R}x\leq_{\R}1$ being $\Pi_{1}^{0}$ (in addition).  
These functions ensure a smooth treatment of $\R$, $[0,1]$, and $2^{\N}$ and functions between such spaces.  We will \textbf{always} assume that real numbers and $\R\di \R$-functions are given in this way, i.e.\ as in the aforementioned references \cite{kohlenbach2, kohlenbach3}, so as to ensure a smooth treatment.
\end{rem}

\subsection{Extending the scope of Turing computability}\label{sketchy}
In this section, we discuss the extension of Turing computability mentioned in Section \ref{sintro}.  
In particular, we introduce this new concept in Section \ref{nore} and discuss its scope in Section \ref{more}.
The reader will have a basic understanding of Turing computability theory (\cite{zweer}) and higher-order primitive recursion like G\"odel's system $T$ (\cite[p.\ 74]{longmann}).

\subsubsection{A new notion of reduction}\label{nore}
In this section, we formulate \eqref{koch5}, which is a relation formalising `is computationally stronger than' involving \textbf{third-order} objects \emph{but} still based on Turing computability. 
We first need some preliminaries, starting with \eqref{koch}. 

\medskip

First of all, many theorems in e.g.\ analysis can be given the form  
\be\label{koch}
(\forall Y: \N^{\N}\di \N)(\exists x \in \N^{\N})A(Y, x), 
\ee
where $\N^{\N}$ is the Baire space and $\N$ is the set of natural numbers.   Indeed, as discussed in Remark \ref{fooker}, some basic 
primitive recursive operations relegate the coding of real numbers (via elements of $\N^{\N}$) to the background.   Moreover, a list of theorems that can be brought in the form \eqref{koch} can be found in Example \ref{exa1} below, while we discuss the scope of theorems that can be brought in this form at the end of this section and in Section \ref{more}. 

\medskip

Secondly, to improve readability, one often uses type theoretic notation in \eqref{koch}, i.e.\ $n^{0}$ for type $0$ objects $n\in \N$, $x^{1}$ for type $1$ objects $x\in \N^{\N}$, and $Y^{2}$ for type $2$ objects $Y:\N^{\N}\di \N$. 
We will only occasionally need type $3$ objects, which map type 2 objects to natural numbers.  We generally use Greek capitals $\Theta^{3}, \Lambda^{3}, \dots $ for such objects.  

\medskip

Thirdly, to compare the logical strength of theorems of the form \eqref{koch}, one establishes results of the following form over weak systems:
\be\label{koch2}
(\forall Y^{2})(\exists x^{1})A(Y, x)\di (\forall Z^{2})(\exists y^{1})B(Z, x),
\ee
as part of Kohlenbach's \emph{higher-order Reverse Mathematics} (see \cite{kohlenbach2} for an introduction).
The computational properties of \eqref{koch} and \eqref{koch2} following S1-S9 can then be studied as follows:  let $\Theta^{3}$ and $\Lambda^{3}$ be \emph{realisers} for the antecedent and consequent of \eqref{koch2}
i.e.\ $(\forall Y^{2})A(Y, \Theta(Y))$ and $(\forall Z^{2})B(Z, \Lambda(Z))$.   

\medskip

A central computability theoretic question concerning \eqref{koch2} is whether a realiser $\Theta^{3}$ for the antecedent of \eqref{koch2} computes, in the sense of S1-S9, a realiser $\Lambda^{3}$ for the consequent of \eqref{koch2}, i.e.\ whether there is a Kleene algorithm with index $e\in \N$ satisfying the following:
\be\label{koch3}
(\forall \Theta^{3})\big[(\forall Y^{2})A(Y,\Theta(Y) )\di (\forall Z^{2})B(Z, \{e\}(\Theta, Z))\big].
\ee
Next, we list some theorems that have been studied via the above paradigm based on \eqref{koch3} and S1-S9.
\begin{exa}[Some representative theorems]\label{exa1}\rm~
\begin{itemize}
\item The Lindel\"of, Heine-Borel, and Vitali covering theorems involving uncountable coverings (\cite{dagsam, dagsamII, dagsamV}),
\item The Lebesgue number lemma (\cite{dagsamV, dagsamVII}),
\item The Baire category theorem (\cite{dagsamVII}),
\item  Convergence theorems for nets (\cite{samnetspilot,samcie19, samwollic19}),
\item Local-global principles like {Pincherle's theorem} (\cite{dagsamV}),
\item The uncountability of $\R$ and the Bolzano-Weierstrass theorem for countable sets in Cantor space (\cite{dagsamX, dagsamXI}),
\item Weak fragments of the Axiom of (countable) Choice (\cite{dagsamIX}).
\item Basic properties of functions of bounded variation, like the Jordan decomposition theorem (\cite{dagsamXII}).
\end{itemize}
Many more theorems are equivalent -in the sense of higher-order RM as in \cite{kohlenbach2}- to the theorems in the above list, as can be found in the associated references. 
\end{exa}
Fourth, for all the reasons discussed in Section \ref{crit}, we formulate a version of \eqref{koch3} based on Turing computability as follows: 
\be\label{koch5}
(\forall Z^{2}, x^{1})\big[A(t(Z), x)\di  [{\{e\}^{s(Z, x)}\downarrow} \wedge B(Z, \{e\}^{s(Z, x)})]\big],
\ee
where $s^{2\di 1}, t^{2\di 2}$ are terms of G\"odel's $T$ and `$\{e\}^{X}$' is the $e$-th Turing machine with oracle $X\subset \N$. 
We note that \eqref{koch5} readily\footnote{For $e\in \N$ and $s^{2\di 1}, t^{2\di 2}$ as in \eqref{koch5}, define $e_{0}\in \N$ as the Kleene algorithm such that $\{e_{0}\}(\Theta, Z):= \{e\}^{s(Z,\Theta(t(Z)))}$, which is total by assumption.} implies \eqref{koch3}; we discuss the generality of \eqref{koch5} at the end of this section. 

\medskip

In line with the nomenclature of computability theory, we call the antecedent and consequent of \eqref{koch2} `problems' and say that 
\begin{center}
\emph{solving the problem} $ (\forall Z^{2})(\exists y^{1})B(Z, x)$ \emph{\textbf{$N$-reduces} to solving the problem} $(\forall Y^{2})(\exists x^{1})A(Y, x)$
\end{center}
in case \eqref{koch5} holds for the parameters mentioned.  
We view the $N$-reduction relation as `neutral' between the Turing and Kleene framework and the reader readily verifies that $N$-reduction is transitive.  
In case the term $s(Z, x)$ can be replaced by a term $u(x)$, i.e.\ the latter has no access to $Z$, we refer to \eqref{koch5} as \textbf{strong} $N$-reduction.  

\medskip

Finally, the critical reader may wonder about the generality of \eqref{koch5}.  
The latter is quite general, for the following two reasons.
\begin{itemize}
\item It is an empirical observation based on \cite{dagsam, dagsamII, dagsamIII, dagsamV, dagsamVI, dagsamVII, dagsamIX, dagsamX, dagsamXI, dagsamXII} that \emph{positive} results in S1-S9 computability theory can be witnessed by terms of G\"odel's $T$ of low complexity.  In this light, there is no real loss of generality if we use terms of G\"odel's $T$ as in \eqref{koch5}. 
\item A theorem of (third-order) ordinary mathematics generally has the form \eqref{koch}, \emph{unless} the former implies the existence of a discontinuous function on $\R$.  
In the latter case, an `indirect' treatment is still possible via the so-called \emph{Grilliot's trick}, which we sketch in Section \ref{krilli}.
\end{itemize}
Like the reader, we feel that the second item deserves a more detailed explanation, which is in Section \ref{more}.  Regarding the first item, intellectual honesty compels us to admit that 
many of our S1-S9 results are witnessed by terms of G\"odel's $T$ \emph{additionally involving} Feferman's search operator (already found in Hilbert-Bernays \cite{hillebilly2}) defined for any $f^{1}$ as:
\be\label{frlu2}
\mu(f):=
\begin{cases}
\textup{the least $n^{0}$ such that $f(n)=0$} & (\exists m^{0})(f(m)=0)\\
0 & \textup{otherwise}
\end{cases}.
\ee
While not strictly necessary always, it is convenient to have access to $\mu^{2}$ as we then do not have to worry how spaces like $[0,1]$ or $2^{\N}$ are represented.  
Based on this observation, we introduce the following:
\begin{center}
\emph{solving the problem} $ (\forall Z^{2})(\exists y^{1})B(Z, x)$ \emph{\textbf{$\mu N$-reduces} to solving the problem} $(\forall Y^{2})(\exists x^{1})A(Y, x)$
\end{center}
in case \eqref{koch5} holds for the parameters mentioned except that $t(Z)$ is replaced by $t(Z, \mu^{2})$.  Then `strong' $\mu N$-reduction is defined similarly.  

\medskip

Finally, one could study \eqref{koch5} for other extensions of G\"odel's $T$, e.g.\ involving `minimization' (see \cite[\S5.1.5]{longmann}), but \eqref{koch5} seems more salient.  

\subsubsection{Continuous and discontinuous functionals}\label{more}
We discuss the motivation behind our notion of $N$-reduction and establish its scope.  
To this end, we have to make the following classical case distinction.  
\begin{itemize}
\item {If} a given third-order theorem is \textbf{consistent} with Brouwer's \emph{continuity theorem} that all functions on $\R$ are continuous (\cite{brouw}), {then} we can \textbf{directly} analyse it via $N$-reduction. 
\item If a given third-order theorem implies the existence of a discontinuous function on $\R$, we can \textbf{indirectly} analyse it via $N$-reduction based on {Grilliot's trick}, where the latter is sketched in Section \ref{krilli}.  
\end{itemize}
To make sense of the above, we first sketch the `standard' higher-order generalisation of (second-order) comprehension, exemplified by Kleene's $\exists^{2}$ as in \eqref{frlu}.  
We then discuss another (less famous) formulation of comprehension, called the \emph{neighbourhood function principle} as in Definition \ref{NFP}, a fragment of the Axiom of Choice involving \emph{continuous} choice functions, going back to intuitionistic analysis (\cite{KT, keuzet}).

\medskip

First of all, the commonplace \emph{one cannot fit a round peg in a square hole} has an obvious counterpart in computability theory: a type 2 functional cannot be the oracle of a Turing machine.  
Nonetheless, a \emph{continuous} type $2$ functional can be \emph{represented} by a type $1$ \emph{Kleene associate} as in Definition \ref{hunki}, where we employ the same\footnote{In particular, $\sigma^{0^{*}}$ is a finite sequence in $\N$ with length $|\sigma|$ and we assume the well-known coding of such finite sequences by natural numbers. Moreover, $\overline{f}n$ is the finite sequence $(f(0), \dots, f(n-1))$ for any $f^{1}$ and $n^{0}$, and any $f^{0^{*}}$ in case $|f|\leq n$.} notations as in \cite{kohlenbach4}. Associates \emph{do} `fit' as oracles of Turing machines.  

\begin{defi}[Kleene associate from \cite{kohlenbach4}]\label{hunki}~
\begin{itemize}
\item A function $\alpha^{1}$ is a \emph{neighbourhood function} if 
\begin{itemize}
\item $(\forall \beta^{1})(\exists n^{0})(\alpha(\overline{\beta}n)>0)$ and
\item $(\forall \sigma^{0^{*}}, \tau^{0^{*}} )(  \alpha(\sigma)>0\di \alpha( \sigma*\tau)=\alpha(\sigma) )$.
\end{itemize}
\item A function $\alpha^{1}$ is a \(Kleene\) \emph{associate} for $Y^{2}$ if 
\begin{itemize}
\item $(\forall \beta^{1})(\exists n^{0})(\alpha(\overline{\beta}n)>0)$ and
\item $(\forall \beta^{1}, n^{0})(\textup{$n$ is least s.t.\ }\alpha(\overline{\beta}n)>0 \di \alpha(\overline{\beta}n)=Y(\beta)+1 )$.
\end{itemize}
\end{itemize}
As in \cite[\S4]{kohlenbach4}, we additionally assume that an associate is a neighbourhood function, as the former can readily be converted to the latter. 
\end{defi}
Hence, we should specify that a \textbf{discontinuous} type two functional cannot be the oracle of a Turing machine.  Now, the archetypal example of a discontinuous function is Kleene's quantifier $\exists^{2}$ defined as:  
\be\label{frlu}
(\forall f^{1})\big[~\underline{(\exists n^{0})(f(n)=0)}~\asa \exists^{2}(f)=0  \big].
\ee
Clearly, \eqref{frlu} is the higher-order version of \emph{arithmetical comprehension} (see e.g.\ \cite[III]{simpson2}) stating that $\{ n\in \N : A(n)\}$ exists for arithmetical formulas $A$, which includes the underlined formula in \eqref{frlu}.   We point out \emph{Grilliot's trick}, a method for (effectively) obtaining $\exists^{2}$ from discontinuous functionals on e.g.\ $\R$ or $2^{\N}$, as also discussed in Section \ref{krilli}.
To our own surprise, this kind of effective result is essentially the prototype of \eqref{koch5}, as discussed in Section \ref{koli}.

\medskip

Now, as noted above, Kleene's $\exists^{2}$ can decide the truth of arithmetical formulas.  
In general, for a formula class $\Gamma$, one can study higher-order functionals that decide the truth of formulas $\gamma \in \Gamma$.  Examples are Kleene's quantifiers $\exists^{n}$ (\cite[Def.\ 5.4.3]{longmann}) and the Feferman-Sieg functionals $\nu_{n}$ from \cite[p.\ 129]{boekskeopendoen}, which we shall however not need.  

\medskip  
 
Secondly, we consider the \emph{neighbourhood function principle} $\NFP$ from \cite{troeleke1}, studied in \cite{KT, keuzet} under a different name.  

\bdefi[$\NFP$]\label{NFP}
For any formula $A(n^{0})$, we have
\be\label{durgopop}
(\forall f^{1})(\exists n^{0})A(\overline{f}n)\di (\exists \gamma\in K_{0})(\forall f^{1})A(\overline{f}\gamma(f)), 
\ee
where `$\gamma\in K_{0}$' means that $\gamma^{1}$ is a (total) Kleene associate.  
\edefi
Clearly, \eqref{durgopop} is a fragment of the Axiom of Choice involving continuous choice functions.   Not as obvious is that $\NFP$ is a `more constructive' formulation of the comprehension axiom (see Remark \ref{knfp} below).    
We also note that $\NFP$ involving third-order parameters has the form \eqref{koch}, namely for a formula $A(n^{0}, Y^{2})$ with all parameters shown, \eqref{durgopop} yields
\be\label{gazi} 
(\forall Y^{2})(\exists \gamma^{1})\big[   (\forall f^{1})(\exists n^{0})A(\overline{f}n, Y)\di  [\gamma\in K_{0}\wedge (\forall f^{1})A(\overline{f}\gamma(f))] \big].
\ee
Now, $\NFP$ proves the Lindel\"of lemma (\cite{troeleke1}), inspiring Theorem~\ref{kilo}.

\medskip

Finally, we can combine the above as follows: the case distinction from the beginning of this section distinguishes between whether a given theorem $\mathfrak{T}$ in the language of third-order arithmetic implies the existence of a discontinuous function (on $\R$ or $2^{\N}$), or not.   
It is then an empirical observation based on \cite{dagsam, dagsamII, dagsamIII, dagsamV, dagsamVI, dagsamVII, dagsamIX, dagsamX, dagsamXI, dagsamXII} that for this theorem $\mathfrak{T}$, \textbf{either} the theorem $\mathfrak{T}$ implies the existence of $\exists^{2}$ via the aforementioned Grilliot's trick, \textbf{or} $\mathfrak{T}$ is provable from a fragment of $\NFP$ where $A$ may include third-order parameters.  

\medskip

In the former case, the theorem $\mathfrak{T}$ can be analysed `indirectly' using $N$-reduction, namely via Grilliot's trick, as discussed in Section \ref{disc}.  In case the theorem $\mathfrak{T}$ is provable from a fragment $\NFP$ (with third-order parameters), we can generally bring $\mathfrak{T}$ in the form \eqref{koch} and hence analyse it \emph{directly} via $N$-reduction.  In light of \eqref{gazi}, fragments of $\NFP$ can always be analysed via $N$-reduction.  

\medskip

In conclusion, second-order comprehension has been generalised to higher types in two (more-or-less-known ways) ways, namely as follows.
\begin{itemize}
\item Formulate `characteristic functionals' like $\exists^{2}$ from \eqref{frlu} that decide the truth of certain formulas.  
\item Formulate $\NFP$ as in Definition \ref{NFP} for formulas involving higher-order parameters and variables.  
\end{itemize}
If a given theorem implies the existence of $\exists^{2}$, we can analyse it `indirectly' via $N$-reduction, namely via Grilliot's trick.
If a given theorem is provable from $\NFP$ involving third-order fragments, we can (readily) analyse it via $N$-reduction.  
In other words, {if} a third-order theorem is \textbf{consistent} with Brouwer's \emph{continuity theorem} that all functions on $\R$ are continuous (\cite{brouw}), {then} we can analyse it directly via $N$-reduction. 

\medskip

Finally, we show that $\NFP$ classically follows from comprehension and vice versa, assuming a fragment of the induction axiom. 
\begin{rem}[$\NFP$ and comprehension]\label{knfp}\rm
To obtain $\NFP$ from comprehension modulo coding of finite sequences, let $X$ be such that $\sigma\in X\leftrightarrow A(\sigma^{0^{*}})$ for any finite sequence $\sigma^{0^{*}}$ in $\N$.  Then define $\gamma(\sigma):= |\sigma|+1$ in case $\sigma\in X$, and $0$ otherwise.  Assuming the antecedent of \eqref{durgopop}, this yields a (total) Kleene associate.  By definition, $\gamma$ also satisfies the consequent of \eqref{durgopop}. 

\medskip

To obtain comprehension from $\NFP$, suppose towards a contradiction that comprehension is false, i.e.\ there is some formula $A(n)$ such that 
\be\label{tangk}
(\forall X\subset \N)(\exists n\in \N )\big[ [n\in X\wedge \neg A(n)]\vee [A(n)\wedge n\not \in X]    \big].
\ee
Now apply $\NFP$ to \eqref{tangk} (coding $X\subset \N$ as elements of $2^{\N}$) to obtain $\gamma\in K_{0}$.  The latter has an upper bound $k_{0}\in \N$ on $2^{\N}$, i.e.\ $n\in \N$ in \eqref{tangk} is bounded by $k_{0}$.  However, the induction axiom readily proves `finite comprehension' as follows:
\be\label{tangk2}
(\forall k\in \N)( \exists X\subset \N)(\forall n\leq k )\big[ n\in X\asa A(n) \big].
\ee
Hence, for $k=k_{0}+1$, \eqref{tangk2} yields a contradiction. 
\end{rem}

\subsection{The need for an extension of Turing computation}\label{crit}
We argue why the extension of Turing computation sketched in Section \ref{sketchy} is necessary and even most welcome, as follows.  
\begin{itemize}
\item Higher-order objects are `coded' as reals so as to accommodate their study via Turing machines.  
It has recently been established that this `coding practise' yields very different results compared to Kleene's approach, even for basic objects like \emph{functions of bounded variation} (Section \ref{plop})
\item The conceptual complexity of Kleene's extension of Turing computability is considerable, while the extension to `infinite time' Turing machines is too general for our purposes (Section~\ref{ohc}).  
\end{itemize}
Put another way, $N$-reduction is an attempt at formulating a relation `is computationally stronger than' for third-order statements that overcomes the above pitfalls, namely the conceptual complexity of Kleene's S1-S9 and the problems associated with second-order representations.

\subsubsection{Computing with second-order representations}\label{plop}
We show that there are huge differences between `computing with higher-order objects' and `computing with \emph{representations of higher-order objects}', even for 
basic objects like functions of bounded variation on $[0,1]$. 

\medskip

Now, various\footnote{Examples of such frameworks are: reverse mathematics (\cite{simpson2, stillebron}), constructive analysis (\cite[I.13]{beeson1}, \cite{bish1}), predicative analysis (\cite{littlefef}), and computable analysis (\cite{wierook}).  Bishop's constructive analysis is not based on Turing computability \emph{directly}, but one of its `intended models' is (constructive) recursive mathematics, as discussed in \cite{brich}.  One aim of Feferman's predicative analysis is to capture  Bishop's approach.} research programs have been proposed in which higher-order objects are \emph{represented/coded} as real numbers or similar representations, so as to make 
them amenable to the Turing framework.  It is then a natural question whether there is any significant difference\footnote{The \emph{fan functional} constitutes an early \emph{natural} example of this difference: it has a computable code but is not S1-S9 computable (but S1-S9 computable in Kleene's $\exists^{2}$ from Section \ref{more}).  The fan functional computes a modulus of uniform continuity for continuous functions on Cantor space; details may be found in \cite{longmann}.\label{seeyouwell}} between the Kleene S1-S9 approach or the Turing-approach-via-codes.     

\medskip

Continuous functions being well-studied$^{\ref{seeyouwell}}$ in this context, Dag Normann and the author have investigated \emph{functions of bounded variation}, which have \textbf{at most} countably many points of discontinuity (\cite{dagsamXII}).  A central result is the \emph{Jordan decomposition theorem} which implies that $f:[0,1]\di \R$ of bounded variation on $[0,1]$ satisfies $f=g-h$ on $[0,1]$ for monotone $g, h:[0,1]\di \R$.  We have the following results. 
\begin{itemize}
\item In case $f:[0,1]\di \R$ of bounded variation is given via a \emph{second-order} representation, then the monotone $g, h:[0,1]\di \R$ such that $f=g-h$, 
can be computed from finite iterations of the Turing jump with $f$ as a parameter by \cite[Cor. 10]{kreupel}. 
\item A \emph{Jordan realiser} $\J$ takes as input $f:[0,1]\di \R$ of bounded variation and outputs $\J(f)=(g,h)$, i.e.\ monotone $g, h:[0,1]\di \R$ with $f=g-h$ on $[0,1]$.  
No Jordan realiser is computable (S1-S9) in any type 2 functional by \cite[Theorem~3.9]{dagsamXI}.
\end{itemize}
Regarding the second item, a Jordan realiser is therefore not computable from (finite iterations of) $\exists^{2}$, the higher-order counterpart of the Turing jump.  
The same holds for $\SS_{k}^{2}$, which is a type two functional that can decide $\Pi_{k}^{1}$-formulas (involving first- and second-order parameters).
The usual proof of the Jordan decomposition theorem implies that Kleene's $\exists^{3}$ computes a Jordan realiser.  But $\exists^{3}$ implies full second-order arithmetic, and the same holds for the combination of all $\SS_{k}^{2}$.  

\medskip

In conclusion, there is a \textbf{huge} difference in the computational hardness of the Jordan decomposition theorem depending on whether we use representations or not.  
However, this theorem deals with functions of bounded variation, a class `very close' to the class of continuous functions. 
Hence, (Turing) computing with representations, interesting as it may be, is completely different from (Kleene) computing with actual higher-order objects.  
In this light, there is a clear need for a notion like $N$-reduction that allows us to compute with actual higher-order objects while staying close to Turing computability. 

\subsubsection{On higher-order computation}\label{ohc}
We argue that the conceptual complexity of Kleene's S1-S9 is considerable, while the extension to `infinite time' Turing machines is too general for our purposes (Section \ref{ohc}).  

\medskip

First of all, as noted above, Turing's famous `machine' model constitutes the first intuitively convincing framework for \emph{computing with real numbers} (\cite{tur37}) while Kleene's S1-S9 extend Turing's approach to \emph{computing with objects of any finite type} (\cite{kleeneS1S9, longmann}).  

\medskip

We have studied or made extensive use of Kleene's S1-S9 computability theory in \cite{dagsam, dagsamII, dagsamIII, dagsamV, dagsamVI, dagsamVII, dagsamIX, dagsamX, dagsamXI, dagsamXII}.
In our opinion, while vastly more general in scope, Kleene's S1-S9 has the following conceptual drawbacks.
\begin{itemize}
\item Turing computability boasts the elementary `Kleene $T$-predicate' (see e.g.\ \cite[p.\ 15]{zweer}) where $T(e, x, y)$ intuitively expresses that $y$ codes the computation steps of the $e$-th Turing machine program with input $x$.  There is no such construct for S1-S9.  
\item Kleene's \emph{recursion theorem} is one of the most elegant and important results in Turing computability (\cite[p.\ 36]{zweer}) and is derived from first principles.  
By contrast, Kleene's schemes S1-S8 formalise higher-order primitive recursion (only), while S9 essentially hard-codes the recursion theorem for S1-S9. 
\item Natural space and time constraints can be formulated for Turing machines, yielding a canonical complexity theory (\cite{aurora}); to the best of knowledge, no such canonical theory exists for higher-order computation in general or S1-S9 in particular.
\item Even basic questions concerning S1-S9 computability theory can be challenging.  We have formulated a most basic example in Section \ref{flonk} concerning the \emph{uncountability of $\R$}, arguably one of the 
most basic properties of the real numbers, which nonetheless yields very hard problems regarding S1-S9 computability.  
\end{itemize}
In conclusion, the previous items suggest that the much greater scope of S1-S9 comes at the cost of conceptual clarity and causes technical difficulties.
It is then a natural question whether we can find a `sweet spot' between the conceptual clarity of Turing computability on one hand, and the generality of S1-S9, leading us to $N$-reduction.  

\medskip

Secondly, an \emph{infinite time Turing machine} (ITTM) (\cite{hamkins2}) is a generalisation of Turing computability involving infinite time or space.   
Welsh provides an overview in \cite{welshman} and Dag Normann studies non-montone inductive definitions and the connection to ITTMs in \cite{dagnonmon}.  

\medskip

In particular, Normann shows that ITTMs can outright compute many of the functionals introduced in \cite{dagsam, dagsamIII, dagsamV}, including realisers for the covering lemmas due to Vitali, Heine-Borel, and Lindel\"of.
However, all these functionals are not S1-S9 computable in any type two functional, i.e.\ the former are `hard to compute' (see \cite{dagsam, dagsamIII, dagsamV}).  
As a result, ITTMs yield `too strong' a baseline framework for our purposes. 

\section{Some results}\label{main}
We establish some results based on our freshly minted notion of $N$-reduction from Section \ref{sketchy}, namely concerning the following topics.
\begin{itemize}
\item Convergence theorems for nets (Sections \ref{conetx} and \ref{cov}).
\item Covering theorems (Sections \ref{cov}).
\item The uncountability of $\R$ (Section \ref{flonk}).
\item Discontinuous functions on $\R$ and Grilliot's trick (Section \ref{disc}).  
\end{itemize}
The below just constitutes an illustrative first collection of examples: we do not claim our results to be particularly deep or ground-breaking.  
We do point out that the above items yield functionals that are, like the Jordan realisers from Section \ref{plop}, hard to compute in that no type 2 functional can (S1-S9) compute them, while $\exists^{3}$ can.  

\medskip

Finally, the curious reader of course wonders what the counterpart of the Turing jump is for $N$-reduction.  We believe this to be the `$J$' operation discussed in Section \ref{conetx}. 

\subsection{Nets and computability theory}\label{netz}
We study basic properties of \emph{nets} via $N$-reduction.  Nets are a generalisation of sequences, and the latter hark back to the early days of computability theory (\cite{specker}).
Filters provide an alternative to nets, but will not be discussed here for reasons discussed in Remark \ref{kutfilter}. 
\subsubsection{Nets, a very short introduction}
Nets are the generalisation of the concept of \emph{sequence} to possibly uncountable index sets, nowadays called \emph{nets} or \emph{Moore-Smith sequences}.  
These were first described in \cite{moorelimit2} and then formally introduced by Moore and Smith in \cite{moorsmidje} and by Vietoris in \cite{kliet}.
These authors also established the generalisation to nets of various basic theorems due to Bolzano-Weierstrass, Dini, and Arzel\`a (\cite[\S8-9]{moorsmidje} and \cite[\S4]{kliet}).

\medskip

One well-know application is the formulation of fundamental topological notions like compactness in terms of nets, as pioneered in \cite{berkhof}, while Kelley's textbook \cite{ooskelly} is standard. 
Tukey's monograph \cite{tukey1} builds a similar framework, based on very specific nets, called \emph{phalanxes}, where the index sets consist of finite subsets ordered by inclusion. 
We now list some basic definitions. 
\bdefi\label{nets}
A set $D\ne \emptyset$ with a binary relation `$\preceq$' is \emph{directed} if
\begin{enumerate}
 \renewcommand{\theenumi}{\alph{enumi}}
\item $\preceq$ is transitive, i.e.\ $(\forall x, y, z\in D)([x\preceq y\wedge y\preceq z] \di x\preceq z )$,
\item for $x, y \in D$, there is $z\in D$ such that $x\preceq z\wedge y\preceq z$, \label{bulk}
\item $\preceq$ is reflexive, i.e.\ $(\forall x\in D)(x \preceq x)$.  
\end{enumerate}
For a directed set $(D, \preceq)$ and a topological space $X$, any mapping $x:D\di X$ is a \emph{net} in $X$.  
We denote $\lambda d. x(d)$ as `$(x_{d})_{d\in D}$' or `$x_{d}:D\di X$' to suggest the connection to sequences.  
The directed set $(D, \preceq)$ is not always explicitly mentioned together with a net $x_{d}:D\di X$.
\edefi 
The following definitions readily generalise from the sequence notion. 
\bdefi[Convergence of nets]\label{convnet}
If $x_{d}:D\di X$ is a net, we say that it \emph{converges} to the limit $\lim_{d} x_{d}=y\in X$ if for every neighbourhood $U$ of $y$, there is $d_{0}\in D$ such that for all $e\succeq d_{0}$, $x_{e}\in U$. 
\edefi
\bdefi[Increasing nets]\label{inc}
A net $x_{d}:D\di \R$ is \emph{increasing} if $a\preceq b$ implies $x_{a}\leq_{\R} x_{b} $ for all $a,b\in D$.
\edefi
Now, we shall mostly use nets where the index set consists of finite sets of real numbers ordered by inclusion, i.e.\ Tukey's `phalanxes' from \cite{tukey1}.    
As noted in Remark \ref{fooker}, real numbers can readily be represented via elements of Baire space using primitive recursive operations.  
Thus, such phalanxes are essentially nets indexed by $\N^{\N}$.  The notion of `sub-sequence' of course generalises to `sub-net' (see e.g.\ \cite{samnetspilot}), but we do not need this (slightly technical) notion here. 

\medskip

Finally, we discuss an alternative to nets and why it is not suitable here.
\begin{rem}[Nets and filters]\label{kutfilter}\rm
For completeness, we discuss the intimate connection between \emph{filters} and nets.  
Now, a topological space $X$ is compact if and only if every \emph{filter base} has a \emph{refinement} that converges to some point of $X$, which follows by \cite[Prop.\ 3.4]{zonderfilter}.  

\medskip

Whatever the meaning of the previous italicised notions, the similarity to the Bolzano-Weierstrass theorem for nets is obvious, and not a coincidence: for every net $\mathfrak{r}$, there is an associated filter base $\mathfrak{B(r)}$ such that if the erstwhile converges, so does the latter to the same point; 
one similarly associates a net $\mathfrak{r(B)}$ to a given filter base $\mathfrak{B}$ with the same convergence properties (see \cite[\S2]{zonderfilter}).  

\medskip

Hence, filters provide an alternative to nets, but we have chosen to work with nets for the following reasons, where the second one is the most pressing.
\begin{itemize}
\item Nets have a greater intuitive clarity compared to filters, in our opinion, due to the similarity between nets and sequences. 
\item Nets are `more economical' in terms of ontology: consider the aforementioned filter base $\mathfrak{B(r)}$ associated to the net $\mathfrak{r}$.  
By \cite[Prop.\ 2.1]{zonderfilter}, the base has strictly higher type than the net.  The same holds for $\mathfrak{r(B)}$ versus $\mathfrak{B}$.  
\item The notion of \emph{refinement} mirrors the notion of sub-net (\cite[\S2]{zonderfilter}).  The former is studied in \cite{sahotop} in the context of paracompactness; the associated results suggest that the notion of sub-net works better in weak systems. 
\end{itemize}
On a conceptual note, the well-known notion of \emph{ultrafilter} corresponds to the equivalent notion of \emph{universal net} (\cite[\S3]{zonderfilter}).  
On a historical note, Vietoris introduces the notion of \emph{oriented set} in \cite[p.~184]{kliet}, which is exactly the notion of `directed set'.  He proceeds to prove (among others)
a version of the Bolzano-Weierstrass theorem for nets.  Vietoris also explains that these results are part of his dissertation, written in the period 1913-1919, i.e.\ during his army service for the Great War.
\end{rem}

\subsubsection{Nets and convergence}\label{conetx}
We obtain a first result concerning $N$-reduction and convergence theorems for nets.
In particular, as promised above, we connect the latter to the following operation, which is central and seems to play the role of the Turing jump:  for given $Y^{2}$, define 
\[
J(Y):=\{ n\in \N: (\exists f^{1})(Y(f, n)=0)   \}.
\]
We now have Theorem \ref{ploiu} where $C$ is Cantor space ordered via the lexicographic ordering $\leq_{\lex}$, i.e.\ the notion of `increasing net in $C$' is obvious following Definition \ref{inc}. 
We note that subsets of $\N^{\N}$ or $\R$ are given by characteristic functions, well-known from measure and probability theory and going back one hundred plus of years (\cite{didi3}). 
\begin{thm}\label{ploiu}
The following strongly $N$-reduce to one and other:
\begin{itemize}
\item for all $Y^{2}$, there is $X\subset \N$ such that $X=J(Y)$,
\item a monotone net in $C$ indexed by Baire space, has a limit. 
\end{itemize}
\end{thm}
\begin{proof}
To show that the second item strongly $N$-reduces to the first one, let $f_{d}:D\di C$ be an increasing net in $C$ indexed by Baire space and consider the formula $(\exists d\in D)(f_{d}\geq_{\lex} \sigma*00\dots)$, where $\sigma^{0^{*}}$ is a finite binary sequence. 
The latter formula is equivalent to a formula of the form $(\exists g^{1})(Y(g, n)=0)$ where $Y$ has the form $t(\lambda d. f_{d}, n)$ for a term $t$ of G\"odel's $T$.  
Now use $J(Y)$ to define the limit $f=\lim_{d}f_{d}$, as follows: $f(0)$ is $1$ if  $(\exists d\in D)(f_{d}\geq_{\lex} 100\dots)$ and zero otherwise.  
One then defines $f(n+1)$ in terms of $\overline{f}n$ in the same way.  Note that we only used $J(Y)$ to define $f$, i.e.\ we have a strong $N$-reduction.  

\medskip

For the remaining case,  fix some $Y^{2}$ and let $w^{1^{*}}$ be a sequence of elements in $\N^{\N}$. 
Define $f_{w}:D\di C $ as $f_{w}:=\lambda k.F(w, k)$ where $F(w, k)$ is $1$ if $(\exists i<|w|)(Y(w(i), k)=0)$, and zero otherwise. 
Then $\lambda w^{1^{*}}.f_{w}$ is a monotone net (phalanx) in $C$ indexed by Baire space (modulo coding).   
In case $\lim_{w}f_{w}=f$, then it is readily verified that:
\be\label{nogisnekeer}
(\forall n^{0})\big[(\exists g^{1})(Y(g,n)=0)\asa f(n)=1\big].
\ee
In the notation of \eqref{koch5}, the net $\lambda w^{1^{*}}f_{w}$ has the form $t(Y)(w)$ while $s$ does not depend on $Y$, i.e.\ we have a strong $N$-reduction.  
\qed
\end{proof}
The reader is warned that not all $N$-reduction results are as elegant.

\subsubsection{Nets and compactness}\label{cov}
We connect the Heine-Borel theorem and convergence theorems for nets via $N$-reduction.

\medskip

First of all, the Heine-Borel theorem, aka \emph{Cousin's lemma}, (\cite{cousin1, opborrelen2}) pertains to open-cover compactness, which we study for the unit interval.
Clearly, each $\Psi:[0,1]\di \R^{+}$ yields a `canonical' covering $\cup_{x\in [0,1]}B(x, \Psi(x))$, which must have a finite sub-covering.  This yields the principle $\HBU$, which has the form \eqref{koch}.
\be\tag{$\HBU$}
(\forall \Psi:[0,1]\di \R^{+})(\exists x_{0}, \dots, x_{k}\in [0,1] )\big( [0,1]\subset \cup_{i\leq k}B(x_{i}, \Psi(x_{i})) \big).
\ee
The reals in $\HBU$ are hard to compute (S1-S9) in terms of $\Psi$, as shown in \cite{dagsam, dagsamII}, as no type two functional can perform this task. 
Computing a \emph{Lebesgue number}\footnote{The notion of \emph{Lebesgue number} is familiar from topology (see e.g.\ \cite[p.\ 175]{munkies}) and amounts to the following: for a metric space $(X,d)$ and an open covering $O$ of $X$, the real number $\delta >0$ is a Lebesgue number for $O$ if every subset $Y$ of $X$ with $\textsf{diam}(Y):=\sup_{x, y\in Y}d(x, y)<\delta$ is contained in some member of the covering.} is similarly hard as shown in \cite{dagsamV}.  
Nonetheless, $\HBU$ seems stronger than the \emph{Lebesgue number lemma} expressing that a Lebesgue number exists for any $\Psi:[0,1]\di \R^{+}$.  
We believe that Theorem \ref{fabuki} expresses this fundamental difference. 
\begin{thm}\label{fabuki}~
\begin{itemize}
\item $\HBU$ $\mu N$-reduces to: for a monotone convergent net in $[0,1]$ indexed by Baire space, there is a modulus\footnote{A modulus of convergence for a net $x_{d}:D\di \R$ with $\lim_{d}x_{d}=x$ is a sequence $(d_{k})_{k\in \N}$ with $(\forall k\in \N)(\forall d\succeq d_{n})(|x_{d}-x|<\frac{1}{2^{k}} )$.}  of convergence. 
\item The \emph{Lebesgue number lemma} strongly $\mu N$-reduces to: a monotone net in $[0,1]$ indexed by Baire space, has a limit. 
\end{itemize}
\end{thm}
\begin{proof}
For the first part, fix $\Psi:[0,1]\di \R^{+}$ and define the following where $w^{1^{*}}$ is a finite sequence of reals:
\[
x_{w}:=
\begin{cases}
1 & (\forall q\in \Q\cap [0,1])( q\in \cup_{i<|w|}B(w(i), \Psi(w(i))) )\\
B(w)/2  & \textup{ otherwise}
\end{cases}.
\]
Here, $B(w)$ is the left-most end-point in $[0,1]$ of the intervals of the form $B(w(i), \Psi(B(w(i))))$ for $i\leq k$ that is not covered by the union.  
Note that $B(w)$ and $x_{w}$ are readily defined using $\mu^{2}$.   Modulo coding of reals, $\lambda w^{1^{*}}.x_{w}$ can be viewed as a monotone net (phalanx) indexed by Baire space and we must have $\lim_{w}x_{w}=1$.  
If $(w_{k})_{k\in \N}$ is a modulus of convergence, then $|x_{w_{2}}-1|<\frac{1}{4}$ by definition, implying $x_{w_{2}}=1$.  
Hence, $\cup_{i<|w_{2}|}B(w(i), \Psi(w(i)))$ covers $[0,1]\cap \Q$.  Now adjoin to $w_{2}$ all the points $w_{2}(i)\pm\Psi(w_{2}(i))$ for $i<|w_{2}|$, to obtain a covering of $[0,1]$. 
This `adjoining' takes the form of $s(\Psi, w_{2})$ while $x_{w}$ takes the form $t(\Psi, \mu^{2})(w)$ for terms $s, t$ of G\"odel's $T$, using the notation from \eqref{koch5}.

\medskip

For the second part, replace the output $1$ by $\frac{3}{4}+\frac{1}{2^{N+3}}$ in the first case of $x_{w}$, where $N$ is as follows: adjoin to $w$ all the points $w(i)\pm\Psi(w(i))$ for $i<|w|$, to obtain a covering of $[0,1]$.  Now use $\mu^{2}$ to find $N\in \N$ such that $\frac{1}{2^{N}}$ is a Lebesgue number for the latter covering.  
Note that the modified net is still monotone as extending $w$ can only increase the associated Lebesgue number. 
Clearly, any cluster point of the modfied net is found in $(\frac{3}{4}, 1)$.  
A straightforward unbounded search can now recover a Lebesgue number from the cluster point of the net \emph{without} access to $\Psi$, i.e.\ we have a strong $\mu N$-reduction. 
\qed
\end{proof}
In light of the first part of the previous proof, the `post-processing' term $s$ in \eqref{koch5} seems necessary as a Turing machine cannot evaluate a third-order functional at a given point due to type restrictions.  

\medskip

As shown in \cite{samnetspilot}, the existence of a modulus of convergence as in the first item of the theorem requires a fragment of the Axiom of Choice ($\AC$) beyond $\ZF$.  
In fact, one readily shows that the former existence statement $N$-reduces (and vice versa) to the following fragment of $\AC$:
\[
(\forall Y^{2})\big[ (\forall n^{0})(\exists f^{1})(Y(f, n)=0)\di (\exists Z^{0\di 1})(\forall n^{0})(Y(Z(n), n)=0) \big], 
\]
where we exclude the trivial case $(\exists f^{1})(\forall n^{0})(Y(f, n)=0)$. 

\medskip

Finally, we connect the Lebesgue number lemma and $\NFP$ as follows.
\begin{thm}\label{kilo}
The \emph{Lebesgue number lemma} strongly $N$-reduces to $\NFP$ for $A(n)\equiv (\exists f^{1})(Y(f, n)=0)$ for any $Y^{2}$.  
\end{thm}
\begin{proof}
By Remark \ref{fooker}, quantifying over $2^{\N}$ or $[0,1]$ amounts to nothing more than quantifying over Baire space.   
To see this, define $\bb:\N^{\N}\di 2^{\N}$ as follows: $\bb(f)(n):=0$ if $f(n)=0$, and $1$ otherwise.  Also, define $\rr(f):=\sum_{n=0}^{\infty} \frac{\bb(f)(n)}{2^{n}}$ as the real in $[0,1]$ coded by $f\in\N^{\N}$. 
For $\Psi:\R\di \R^{+}$, the following formula is trivial (take $g=f$ and large $n$):
\[\textstyle
(\forall f\in \N^{\N})(\exists n\in \N)\big[(\exists g\in \N^{\N})[ B(\rr(f), \frac{1}{2^{n}}) \subset B(\rr(g), \Psi(\rr(g)))   ]   \big], 
\]
which merely expresses that for every $x\in [0,1]$, there is $n\in \N$ and $y\in [0,1]$ such that $B(x, \frac{1}{2^{n}})\subset B(y, \Psi(y))$.
Applying $\NFP$ with parameter $\Psi$, we obtain $\gamma\in K_{0}$ such that 
\[\textstyle
(\forall f\in \N^{\N})\big[(\exists g\in \N^{\N})[ B(\rr(f), \frac{1}{2^{\gamma(f)}}) \subset B(\rr(g), \Psi(\rr(g)))   ]   \big].
\]
Now compute an upper bound for $\gamma$ on $2^{\N}$, using the Kleene associate for the fan functional (\cite[\S8.3.2]{longmann}).  
This upper bound yields the required Lebesgue number, which only depends on $\gamma^{1}$, not on $\Psi$, i.e.\ we have obtained a \emph{strong} $N$-reduction. 
\qed
\end{proof}
We conjecture that $\HBU$ does not strongly $\mu N$-reduce to the fragment of $\NFP$ from Theorem \ref{kilo}.

\subsection{On the uncountability of $\R$}\label{flonk}
We study one of the most (in)famous properties of $\R$, namely its uncountability, established by Cantor in 1874 as part of his/the first set theory paper (\cite{cantor1}).
The following two principles were first studied in \cite{dagsamX, dagsamXI}. 
\begin{itemize}
\item $\NIN$: there is no injection from $[0,1]$ to $\N$.  
\item \textbf{Cantor's theorem}: for a set $A\subset [0,1]$ and $Y:[0,1]\di \N$ injective on $A$, there is $x\in \big([0,1]\setminus A)$.
\end{itemize}
A trivial manipulation of definitions shows that $\NIN$ and Cantor's theorem are logicially equivalent.  We however have the following theorem and 
associated Conjecture \ref{durfke}. 
\begin{thm}~
\begin{itemize}
\item The problem $\NIN$ $N$-reduces to the Heine-Borel theorem $\HBU$. 
\item Cantor's theorem $N$-reduces to the Heine-Borel theorem $\HBU$ restricted to Baire class $2$ functions.
\end{itemize}
\end{thm}
\begin{proof}
For the first part, fix $Z:[0,1]\di \N$ and define $t(Z)(x):=\frac{1}{2^{Z(x)+1}}$ motivated by the notation in \eqref{koch5}.  
In case $x_{0}, \dots, x_{k}\in [0,1]$ is a finite sub-covering of $\cup_{x\in [0,1]}B(x, t(Z)(x))$, there are $i, j\leq k$ with 
\be\label{ponk}
Z(x_{i})=Z(x_{j})\wedge x_{i}\ne x_{j}.
\ee
Indeed, in case there are no $i, j\leq k$ as in \eqref{ponk}, then the measure of $\cup_{i\leq k} B(x_{i}, t(Z)(x_{i}))$ is at most $\sum_{n=0}^{k}\frac{1}{2^{i+1}}<1$, contradicting the fact that $\cup_{i\leq k}B(x_{i}, t(Z)(x_{i}))$ covers $[0,1]$. 
In light of \eqref{ponk}, given the finite sequence $s(Z, x_{0}, \dots, x_{k})$ defined as $ x_{0}, t(Z)(x_{0}), \dots, x_{k}, t(Z)(x_{k})$, we can perform an unbounded search (on a Turing machine) to find $i, j\leq k$ and $k\in \N$ such that $t(Z)(x_{i})=_{\Q} t(Z)(x_{j})$ and $[|x_{i}-x_{j}|](k)>_{\Q}\frac{1}{2^{k}} $, where $[z](m)$ is the approximation of $z\in \R$ up to $\frac{1}{2^{m+1}}$.
Hence, we also obtain the consequent of \eqref{koch5} for the case at  hand. 

\medskip

For the second part, fix $A\subset [0,1]$ and $Y:[0,1]\di \N$ such that $Y$ is injective on $A$.  
Now consider the following:
\[
t(Y, A)(x):=
\begin{cases}
\frac{1}{2^{Y(x)+5}} & x\in A\\
\frac{1}{8} & x \not \in A
\end{cases}.
\]
One readily shows that $t(Y, A):\R\di \R$ is Baire class 2, as it only has countably many points of discontinuity by definition.  
For a finite sub-covering $x_{0}, \dots, x_{k}\in [0,1]$ of $\cup_{x\in [0,1]}B(x, t(Y, A)(x))$, there must be $j\leq k$, with $x_{j}\not \in A$.  
Indeed, as in the previous paragraph, the measure of $\cup_{i\leq k} B(x_{i}, t(Y, A)(x_{i}))$ is otherwise at most $\sum_{n=0}^{k}\frac{1}{2^{i+5}}<1$, a contradiction. 
One can effectively decide whether $t(Y, A)(x_{i})<\frac{1}{8}$ or $t(Y, A)(x_{i})>\frac{1}{16}$ for $i\leq k$, i.e.\ one readily finds a $j\leq k$ with $x_{j}\not \in A$.  \qed
\end{proof}
In light of the previous proof, the `post-processing' term $s$ in \eqref{koch5} again seems necessary as a Turing machine cannot evaluate a third-order functional at a point due to type restrictions.  

\medskip

Based on the previous proof, we conjecture the following. 
\begin{conj}\label{durfke}
The problem $\NIN$ does \textbf{not} $N$-reduce to the Heine-Borel theorem $\HBU$ restricted to Baire class $2$ functions, nor to the \(full\) Lebesgue number lemma.
\end{conj}

\subsection{Discontinuous functions}\label{disc}
We show that a representative equivalence from the Reverse Mathematics literature involving $(\exists^{2})$ gives rise to $N$-reductions between the members of the equivalence.  
That $N$-reduction applies here was surprising to us, as the existence of a discontinuous function like $\exists^{2}$ does not have 
the syntactic form \eqref{koch}. 

\medskip

A central role is played by \emph{Grilliot's trick}, a method for (effectively) obtaining $\exists^{2}$ from a discontinuous function (\cite{grilling}).  
We discuss this trick in some detail in Section \ref{krilli}, while the connection between this trick and $N$-reduction is discussed in Section \ref{koli}.

\subsubsection{Grilliot's trick}\label{krilli}
In a nutshell, \emph{Grilliot's trick} is a method for effectively obtaining $\exists^{2}$ from a discontinuous function, say on $\N^{\N}$ or $\R$.  
Clearly, $\exists^{2}$ is discontinuous at $11\dots$, making the former functional a kind of `canonical' discontinuous function. 

\medskip

First of all, Grilliot's paper \cite{grilling} pioneers the aforementioned method, nowadays called Grilliot's trick; we refer to \cite[Remark 5.3.9]{longmann} for a discussion of the general background and history.    
We note that Kohlenbach formalises Grilliot's trick in a weak logical system (namely his `base theory' $\RCAo$) in \cite[\S3]{kohlenbach2}.

\medskip

Secondly, Kohlenbach's rendition of Grilliot's trick (\cite[\S3]{kohlenbach2}) is quite easy to understand conceptually.  Indeed, assume we have a function $F:\R\di \R$ and a sequence $(x_{n})_{n\in \N}$ with $\lim_{n\di \infty}x_{n}=x$ such that $\lim_{n\di \infty}F(x_{n})\ne F(x)$, i.e.\ $F$ is not sequentially continuous at $x$.
Then there is a term $t^{3}$ of G\"odel's $T$ of low complexity such that $E(f):=\lambda f^{1}.t(F, \lambda n. x_{n}, x, f)$ is Kleene's $\exists^{2}$ as in \eqref{frlu}.  
All technical details, including the exact definition of $t$, are found in \cite[\S3]{kohlenbach2}.

\medskip

Thirdly, Kohlenbach uses Grilliot's trick in \cite[\S3]{kohlenbach2} to show that e.g.\ the following sentence implies the existence of $\exists^{2}$:
\be\label{tonk}
\underline{(\exists \varepsilon)(\forall g\in L([0,1]))}[\varepsilon(g) \in [0,1]\wedge (\forall y\in [0,1])(    g(y)\leq g(\varepsilon(g))) ].
\ee
Here, $\varepsilon(g)$ is a real in $[0,1]$ where the Lipschitz-continuous\footnote{A function $g:[0,1]\di \R$ is \emph{Lifschitz-continuous} with constant $1$ on $[0,1]$ if $(\forall x, y\in [0,1])(|g(x)-g(y)|< |x-y|  )$.  Hence, to (effectively) recover the graph of $g$, it suffices to have access to the sequence $(g(q))_{q\in \Q\cap [0,1]}$.\label{doeme}} function $g:[0,1]\di \R$ with constant $1$ attains its maximum.  
We note that the underlined quantifiers can be brought in the form$^{\ref{doeme}}$ `$(\exists Y^{2})(\forall \alpha^{1})$', which can also be obtained by representing continuous functions via second-order codes.

\subsubsection{Discontinuous functions and $N$-reduction}\label{koli}
In this section, we discuss the connection between Grilliot's trick from Section \ref{krilli} and $N$-reduction.  
In particular, we show that the proof of \cite[Prop.\ 3.14]{kohlenbach2}, establishing the equivalence $\eqref{tonk}\asa (\exists^{2})$, gives rise to $N$-reductions involving \eqref{tonk} and $(\exists^{2})$. 

\medskip

First of all, consider \eqref{tonk} from Section \ref{krilli}.  The functional $\eps$ from \eqref{tonk} yields a discontinuous function on $\R$, which yields $\exists^{2}$ in turn, following the proof of \cite[Prop.\ 3.14]{kohlenbach2}. 
If we make all steps in the latter proof explicit\footnote{The construction of a discontinuous function on $\R$ in the proof of \cite[Prop.~3.14]{kohlenbach2} depends on whether $\eps(g_{0})\in [0, \frac{1}{2}]$ or $\eps(g_{0})\in [\frac{1}{2},1]$, where $g_{0}$ is the constant $0$ function and $\eps$ as in \eqref{tonk}.  This non-effective case distinction can be replaced by an effective case distinction whether $\eps(g_{0})<\frac34$ or $\eps(g_{0})>\frac14$. The proof in the first case goes through unmodified, while one replaces $yx$ in the second case by $-yx$.}, we obtain a term $t$ of G\"odel's $T$ of low complexity such that 
\be\label{koleire}
(\forall \eps) \big [(\forall g) A(g, \eps(g))\di   (\forall f^{1})B(t(\eps), f)    \big],  
\ee
where $A(g, x )$ expresses that $x\in [0,1]$ is a real where the Lipschitz-continuous function $g:[0,1]\di \R$ with Lipschitz constant $1$ attains its maximum; the formula $(\forall f^{1})B(\exists^{2}, f)$ is \eqref{frlu}, i.e.\ the specification of $\exists^{2}$. 
Clearly, \eqref{koleire} implies by contraposition that:
\be\label{koleire2}
(\forall \eps, f^{1}) \big [\neg B(t(\eps), f)\di (\exists g) \neg A(g, \eps(g))    \big],
\ee
which is `almost' the definition of $N$-reduction as in \eqref{koch5}.  Indeed, `$(\exists g)$' in \eqref{koleire2} is essentially a quantifier over $\R$ by Footnote \ref{doeme}, whence $(\forall \varepsilon)$ can be viewed as a quantifier $(\forall Z^{2})$. 
Furthermore, a detailed inspection of the proof that \eqref{tonk} implies the existence of $\exists^{2}$ in \cite[Prop.\ 3.14]{kohlenbach2}, reveals the following: this proof still goes through if we restrict \eqref{tonk} to a sentence of the form:
\be\label{flunksd}
{(\exists \varepsilon)(\forall n^{0}))}[   \varepsilon(g_{n}) \in [0,1]\wedge (\forall q\in [0,1]\cap \Q)(    g_{n}(q)\leq g_{n}(\varepsilon(g_{n}))) ],
\ee
for some effective\footnote{The join of the sequences $(qx-q)_{q\in \Q\cap [0,1]}$ and $(-qx)_{q\in \Q\cap [0,1]}$ suffices.} sequence of functions $(g_{n})_{n\in\N}$ all in $L([0,1])$.  
In case the formula in square brackets in \eqref{flunksd} is false for some $n\in \N$, an unbounded search will yield this number. 
Hence, we can replace `$(\exists g) \neg A(g, \eps(g)$' in \eqref{koleire2} by 
\[
{\{e\}^{s(\eps, f)}\downarrow} \wedge\neg A(\{e\}^{s(\eps, f)}, \eps(\{e\}^{s(\eps, f)})\big) 
\]
for some index $e\in \N$ and term $s$ of G\"odel's $T$, which is exactly \eqref{koch5}.
The details are somewhat tedious, but we nonetheless can say that the negation of \eqref{tonk} $N$-reduces to the negation of $(\exists^{2})$. 

\medskip

Finally, the usual `interval-halving' proof of the existence of a maximum of a continuous function on $[0,1]$, can be done using $\exists^{2}$, yielding a term $t$ of G\"odel's $T$ such that:
\be\label{koleire5}
(\forall E^{2}) \big [   (\forall f^{1})B(E, f)\di (\forall g) A(g, t(E)(g))    \big].
\ee
The contraposition of \eqref{koleire5} then has the same form as \eqref{koleire2}.
One readily obtains an index $e\in \N$ and term $s$ of G\"odel's $T$ with
\be\label{koleire6}
(\forall E^{2}, g) \big [   \neg A(g, t(E)(g))\di  [{\{e\}^{s(E, g)}\downarrow} \wedge   \neg B(E, \{e\}^{s(E, g)})  \big],
\ee
as one only needs to decide $g(r)\geq g(q)$ for $r, q\in [0,1]\cap \Q$ to find a maximum of a (Lipschitz) continuous function $g:[0,1]\di \R$.
Hence, an unbounded search on a Turing machine will find $f^{1}$ with $\neg B(E, f)$.  We note that \eqref{koleire6} is a case of $N$-reduction of the negation of $(\exists^{2})$ to the negation of \eqref{tonk}.

\medskip

In conclusion, we observe that the negation of $(\exists^{2})$ will $N$-reduce to the negation of \eqref{tonk}, and vice versa. 
Thus, it perhaps makes sense to drop the `negation of' here and distinguish between \eqref{koch} and its negation in the definition of $N$-reduction.

\begin{ack}\rm
I thank Anil Nerode for his most helpful advise.
My research was kindly supported by the \emph{Deutsche Forschungsgemeinschaft} via the DFG grant SA3418/1-1.  
I thank the anonymous referees for their suggestions, which have greatly improved this paper.  
\end{ack}

\section*{Bibliography}


\end{document}